\documentclass[a4paper,11pt]{amsart}
\usepackage[margin=2cm]{geometry}
\usepackage[utf8]{inputenc}
\usepackage[english]{babel}
\usepackage{amssymb, stmaryrd}
\usepackage{amsmath}
\usepackage{amsthm}
\usepackage{amsfonts}
\usepackage{array}
\usepackage{comment}
\usepackage{enumitem}
\usepackage{hyperref}
\usepackage{mathrsfs}
\usepackage{tikz}
\usepackage{tikz-cd, mathtools}
\usepackage{mathabx}
\usepackage{xfrac}
\usepackage{graphicx} % Required for inserting images
\usepackage[mathscr]{euscript}

\theoremstyle{plain}
\newtheorem{theorem}{Theorem}[section]
\newtheorem{lemma}[theorem]{Lemma}

\newtheorem{proposition}[theorem]{Proposition}
\newtheorem{corollary}[theorem]{Corollary}

\theoremstyle{definition}

\newtheorem{example}[theorem]{Example}
\newtheorem{remark}[theorem]{Remark}

\newcommand{\A}{\mathbb{A}}
\newcommand{\N}{\mathbb{N}}
\newcommand{\C}{\mathbb{C}}

\newcommand{\CP}{\mathscr{P}}

\newcommand{\CT}{\mathcal{T}} % family polytope al

\newcommand{\pa}{T} % polytope algebra

\newcommand{\te}{\widetilde{e}}
\newcommand{\tf}{\widetilde{f}}
\renewcommand{\P}{\mathbb{P}}

\newcommand{\R}{\mathbb{R}}
\newcommand{\Z}{\mathbb{Z}}

\DeclareMathOperator{\cha}{char}
\DeclareMathOperator{\SYT}{SYT}

\DeclareMathOperator{\Spec}{Spec}
\DeclareMathOperator{\Proj}{Proj}

\DeclareMathOperator{\maj}{maj}
\DeclareMathOperator{\des}{des}
\DeclareMathOperator{\bij}{Bij}
\DeclareMathOperator{\Des}{Des}
\newcommand{\ba}{{\bf a}}
\newcommand{\bv}{{\bf v}}

\title[The projective coinvariant algebra]{The projective coinvariant algebra, Young invariants and \\ bigraded coordinate rings of Segre embeddings}
\author{Bal\'{a}zs Szendr\H{o}i}
\address{University of Vienna, Austria}
\email{balazs.szendroi@univie.ac.at}

\begin{document}

\begin{abstract}\noindent This paper studies a flat degeneration~$P_n$ of the classical coinvariant algebra~$R_n$, a bigraded Artinian Gorenstein algebra that arises from the coordinate ring of the Segre embedding of the $n$-fold self-product of the projective line. The Frobenius character of $P_n$ is computed by a natural bigraded refinement of the classical Lusztig--Stanley formula for the character of the coinvariant algebra. Young invariants in~$P_n$ get related to coordinate rings of general Segre embeddings of products of projective spaces; their bigraded Hilbert polynomials get expressed in terms of major-descent generating functions of words in multisets. Relations to the diagonal coinvariant algebra, cohomological interpretations including quantum cohomology, and Garsia--Stanton--style bases are also explored. 
\end{abstract}
\maketitle

\thispagestyle{empty}

\section{Introduction}

Given a positive integer $n$, fixed for the remainder of the paper, let $S_n$ denote the symmetric group on $\{1,\ldots, n\}$. The group $S_n$ acts on the free polynomial ring $\C[u_1\,\ldots, u_n]$ in the standard way by permuting the indices. 
The classical {\em coinvariant algebra}
\[ R_n = \C[u_1\,\ldots, u_n]/\langle \C[u_1\,\ldots, u_n]^{S_n}_+\rangle
\]
is a well-studied Artinian graded complete intersection $\C$-algebra of dimension $n!$ and Hilbert series 
\begin{equation}
 \label{eq:Rnqdim}   
\dim_{q} R_n  = \sum_{\sigma\in S_n} q^{\mathrm{maj}(\sigma)} = [n]_q!, 
\end{equation}
with $\mathrm{maj}$ the major index statistic on permutations of $n$ (see~\ref{subsec:conv}-\ref{subsec:multi} for all 
undefined terms). The Garsia--Stanton $\C$-basis $\{\mathbf{b}_\sigma\colon \sigma\in S_n\}\subset R_n$ from~\cite{GS} witnesses formula~\eqref{eq:Rnqdim} in the sense that it associates to every permutation $\sigma\in S_n$ a homogeneous basis element $\mathbf{b}_\sigma$ of graded degree $\maj(\sigma)$. 
Further, the algebra~$R_n$ carries a graded action of the symmetric group $S_n$, which makes it isomorphic to the regular representation of $S_n$ as ungraded representation. Its graded (Frobenius) character is computed by the Lusztig--Stanley formula~\cite[Prop. 4.11]{Stan}
\begin{equation} \label{eq:StanLusz}\cha_{q} R_n = \sum_{\lambda\vdash n} \left(\sum_{T\in\SYT(\lambda)} q^{\mathrm{maj}(T)}\right)s_\lambda,
\end{equation}
which expresses the $q$-multiplicity of the irreducible representation $s_\lambda$ of $S_n$ in $R_n$ as a sum over standard Young tableaux (see~\ref{subsec:SYT}) of shape equal to a given partition $\lambda\vdash n$. 

Given a composition $\alpha=(\alpha_1, \ldots, \alpha_k) \vDash n$ of $n$, an additive decomposition of $n$ into positive integers, there is a corresponding Young subgroup  \[S_\alpha=S_{\alpha_1} \times \ldots \times S_{\alpha_k}\subset S_n.\]
The action of the symmetric group $S_n$ on the graded algebra $R_n$ restricts to an action of $S_\alpha$. The invariant algebra
\[ R_\alpha = R_n^{S_\alpha}, 
\]
the {\em partial coinvariant algebra}~\cite[Sect.~5]{BGG}, \cite[Sect.~2]{MR}, is a graded finite-dimensional $\C$-algebra with Hilbert series
\begin{equation}\label{eq:dimqRalpha} \dim_q R_\alpha = \left[\!\!\begin{array}{c}n\\\alpha\end{array}\!\!\right]_q,
\end{equation}
a $q$-multinomial coefficient~\cite[Lemma 2.1]{Br}. 

In this paper, I study a certain bigraded degeneration of the above setup, related to Segre embeddings of products of projective spaces. The main object is a certain bigraded algebra $P_n$ that I will call the {\em projective coinvariant algebra}, which was introduced independently at least three times
in the literature in~\cite{ORS, BO, AS}, but in a way that I feel has not uncovered all its remarkable properties. 
The following set of results will be discussed. 
\begin{enumerate}
\item[(a)] (Proposition~\ref{prop_reg}) The algebra $P_n$ is defined as a certain bigraded, Artinian truncation of the coordinate ring
of the $n$-fold product of the projective line $\C[(\P^1)^n]$ in its standard Segre embedding. It is a bigraded, 
finite-dimensional Gorenstein algebra of dimension $n!$ and bigraded Hilbert series
\begin{equation}\label{eq:Pntqdim}
\dim_{t,q} P_n = \sum_{\sigma\in S_n} t^{\mathrm{des}(\sigma)}q^{\mathrm{maj}(\sigma)}, 
\end{equation}
with $\mathrm{des}$ the descent statistic on $S_n$. 

\item[(b)] (Theorem~\ref{thm_main_def}) The algebra $P_n$ is the central fibre of a natural two-dimensional family $\CP_n$ of algebras over the base $\C^2$, 
which carries an action of the torus $(\C^*)^2$ that is compatible with its standard action on the base. Its bigrading
arises from being a fibre over the torus-fixed point $(0,0)\in\C^2$. Over points $(s_0,0), (0,s_1)\in\C^2$ for $s\neq0$, 
the fibres of the family $\CP_n$ are isomorphic to the standard coinvariant algebra $R_n$, with its grading arising from 
the one-dimensional subtorus $\C^*\subset(\C^*)^2$ fixing such a point in the base. Over general points 
$(s_0,s_1)\in\C^2$ with $s_0,s_1\neq0$, the fibres of the family $\CP_n$ are isomorphic to the group algebra $\C S_n$ as ungraded
algebras. 

\item[(c)] (Theorem~\ref{thm:char}) The algebra $P_n$ is a bigraded $S_n$-module with Frobenius character
\begin{equation}\label{eq:Pntqchar}
\cha_{t,q} P_n  =   \sum_{\lambda\vdash n} \left(\sum_{T\in\SYT(\lambda)} t^{\mathrm{des}(T)}q^{\mathrm{maj}(T)}\right)s_\lambda.
\end{equation}

\item[(d)] (Theorem~\ref{thm:invariantisPalpha}(i)-(ii)) The invariant subalgebra $P_\alpha = (P_n)^{S_\alpha}$ under the action of a Young subgroup $S_\alpha\subset S_n$ can be identified
with a certain Artinian truncation of the coordinate ring of a general product of projective
spaces $\C[\P^{\alpha_1}\times\ldots\times\P^{\alpha_k}]$ in its Segre embedding. It is a bigraded, Artinian algebra with bigraded Hilbert series
\begin{equation}\label{eq:Patqdim}
\dim_{t,q}  P_\alpha=  \sum_{w\in W_\alpha} q^{\maj(w)}t^{\des(w)},    
\end{equation}
where $W_\alpha$ is a certain set of words in a multiset associated to $\alpha$, equipped with its standard
descent and major statistics (once again, see Section~\ref{subsec:multi}). 

\item[(e)] (Theorem~\ref{thm:invariantisPalpha}(iii)-(iv)) 
The algebra $P_\alpha$ also has a family of equivariant deformations over the base $\C^2$, with other fibres being isomorphic to 
the partial coinvariant algebra $R_\alpha$, respectively a finite-dimensional ungraded algebra of dimension
equal to the multinomial coefficient ${n\choose \alpha}$.  
\end{enumerate}

The degeneration picture of~(b) explains in particular how the two-parameter expressions \eqref{eq:Pntqdim}-\eqref{eq:Pntqchar} specialize at $t=1$ to the one-parameter expressions~\eqref{eq:Rnqdim}-\eqref{eq:StanLusz}, 
which then further specialise at $q=1$ to $\dim \C S_n = n!$, respectively the standard decomposition of the regular representation $\C S_n$
into irreducibles. On the level of (graded) dimensions, similar comments apply to the case of $S_\alpha$-invariants. 

As an intermediate step in the argument, given a composition $\alpha=(\alpha_1, \ldots, \alpha_k) \vDash n$ of $n$, I establish in Theorem~\ref{thm:bigradedPa} the formula
\[
\dim_{t,q}  \C[\P^{\alpha_1}\times\ldots\times\P^{\alpha_k}]=  \frac{ \displaystyle\sum_{w\in W_\alpha} q^{\maj(w)}t^{\des(w)}}{ \displaystyle\prod_{j=0}^n (1-tq^j)}
\]
for the bigraded Hilbert series of the projective coordinate ring of a general product of projective spaces (in the standard
Segre embedding). This refines in a combinatorial manner a formula (Corollary~\ref{cor:nemcomb}) for the classical Hilbert series (involving only the variable $t$, with $q=1$) that appears in the literature already~\cite{morales}, and is also implicitly proven in~\cite{ORS}. 

Further connections between the constructions of this paper and earlier work include the following. 
\begin{enumerate}
\item[(f)]  It appeared plausible that there should be a connection between the projective coinvariant algebra $P_n$, 
and the Garsia--Haiman diagonal coinvariant algebra $D_n$ (Remark~\ref{rem:PnDn}). Indeed there is one, but it turns out 
to be trivial (Proposition~\ref{prop:PnDn}). A fundamental reason for this disappointing fact eludes me. 
\item[(g)] For $n=k\cdot m$, consider the composition $\alpha=(m,\ldots, m)\vDash n$. There is a residual action
of the symmetric group $S_k$ on the corresponding algebra $P_\alpha$. Inspired by an early version of this work, its bigraded character was recently computed in~\cite{LR}; the resulting nice formula is recalled in Theorem~\ref{thm:fabianetal}. 
\item[(h)] Borel's isomorphism is a graded isomorphism 
\[
R_n \cong H^*(F_n, \C)
\]
between the coinvariant algebra $R_n$ and the cohomology algebra
of the flag variety~$F_n$. For a partial coinvariant algebra, there is a graded isomorphism
\[R_\alpha \cong H^*(F_\alpha, \C)\]
with the cohomology algebra of a partial flag variety $F_\alpha$. In this context, the results presented above can be interpreted as providing simultaneously a (bigraded) degeneration and an (ungraded) deformation (smoothing) of the (graded) cohomology algebra $H^*(F_\alpha, \C)$ of any partial flag variety. 
For one-step flag varieties ($k=2$), the one-parameter smoothing can be identified with the standard quantum deformation of the cohomology of the corresponding flag variety (Proposition~\ref{prop:cohom} and the following discussion).  
I do not know whether the bigraded degenerations $P_n, P_\alpha$ can be interpreted cohomologically, but it would be really interesting if this were to be the case. (Section~\ref{sec:cohom}.) 
\item[(i)] In Section~\ref{sec:bases}, I recall $\C$-bases $\{\mathbf{a}_\sigma\colon\sigma\in S_n\}\subset P_n$, respectively 
$\{\mathbf{a}_w\colon\sigma\in W_\alpha\}\subset P_\alpha$, already constructed by~\cite{BO}, respectively by~\cite{RSW}, that witness formulae~\eqref{eq:Pntqdim}-\eqref{eq:Patqdim} in the same sense as the Garsia--Stanton basis witnesses~\eqref{eq:Rnqdim}. Using the deformation~(e), I deduce the existence of Garsia--Stanton-style bases in partial coinvariant algebras $R_\alpha$, which may be new. 
\end{enumerate}

\subsection*{Acknowledgements} I would like to thank Erik Carlsson, Ionut Ciocan-Fontanine, Iain Gordon, Tam\'as Hausel, Fabi\'an Levi\-c\'an, Jun Ma, Anton Mellit, Konstanze Rietsch, Marino Romero and Dmitriy Rumynin for discussions, comments and help with references, and Praise Adeyemo for an earlier collaboration which initiated this research. 

\section{Preliminaries}
\subsection{Conventions and notations} \label{subsec:conv} We work over the field $\C$. Fix a positive integer $n$.
Denote $[n]=\{1, \ldots, n\}$, and let
% and $\CP[n]$ the set of all its subsets. 
$S_n=\bij[n]$ be the symmetric group on $n$ letters. 
The $q$-integers are defined as usual by 
\[ [n]_q = \frac{q^n-1}{q-1}\]
and the $q$-factorial by
\[ [n]_q! = \prod_{j=1}^n [j]_q.\]
Given a non-negatively graded $\C$-algebra $R=\bigoplus_{d=0}^\infty R_d$ with finite-dimensional graded pieces, define its $q$-dimension (Hilbert series) by the formula
\[\dim_q R = \sum_{d=0}^\infty (\dim R_d) q^d\in{\mathbb N}\llbracket q\rrbracket.
\]
Similarly, for a bigraded algebra $R=\bigoplus_{d,e=0}^\infty R_{d,e}$, denote
\[\dim_{t,q} R = \sum_{d,e=0}^\infty (\dim R_{d,e}) t^dq^e \in{\mathbb N}\llbracket q,t\rrbracket.
\]
If $V$ is a finite-dimensional representation of the symmetric group $S_n$ with irreducible decomposition $V=\oplus_{\lambda\vdash n} V_\lambda$ indexed by partitions $\lambda\vdash n$ of~$n$, define its (Frobenius) character to be
\[\cha V = \sum_{\lambda\vdash n} (\dim V_\lambda)s_\lambda, 
\]
where the $s_\lambda$ are the standard (Schur) basis elements in the representation ring of $S_n$. When $V$ is a 
(bi)graded representation with finite-dimensional graded pieces, the two notations combine in the obvious way. We will sometimes
use standard plethystic notation to write compact expressions for Frobenius characters; see for example~\cite[Sect.3.8]{Ber}. 
For a partition $\lambda\vdash n$, $h_\lambda, m_\lambda$ denote the standard complete and monomial symmetric functions.

\subsection{Words in multisets and an identity of MacMahon} \label{subsec:multi}
Recall that a composition $\alpha \vDash n$ is an ordered sequence $\alpha=(\alpha_1, \ldots, \alpha_k)$ of positive integers summing to $n$, with no assumption on the sequence being arranged in decreasing order (as there would be for a partition). 
Given a composition $\alpha \vDash n$, let $M_\alpha=\{1^{\alpha_1}, \ldots, k^{\alpha_k}\}$ be the corresponding multiset (set with repetition of elements allowed). Let ${\mathcal P}(M_\alpha)$ denote the set of subsets of $M_\alpha$, the set of sub-multisets $\{1^{\beta_1}, \ldots, k^{\beta_k}\}$ of $M_\alpha$ with $0\leq \beta_i\leq\alpha_i$, of size $|{\mathcal P}(M_\alpha)|=\prod_{i=1}^k (\alpha_i+1)$. 

Let $W_\alpha$ be the set of full-length words $w=(i_1i_2\ldots i_n)$ in the multiset $M_\alpha$. The size of this set is 
\[|W_\alpha| = {n\choose \alpha} = \frac{n!}{\alpha_1!\cdot\ldots\cdot\alpha_k!},\]
a multinomial coefficient. For a word $w=(i_1i_2\ldots i_n)\in W_\alpha$, $1\leq k \leq n-1$ is a {\em descent position} if $i_k>i_{k+1}$. Let $\Des(w)$ denote the set of descent positions in $w$, and define the descent index $\des(w)$, respectively major index $\maj(w)$, to be the number, respectively sum, of descent positions in $w$:
\[ \des(w) = |\Des(w)|, \ \ \maj(w) = \sum_{i\in \Des(w)}i.\]
We get the generating function
\[
A_\alpha(t,q) = \sum_{w\in W_\alpha} q^{\maj(w)}t^{\des(w)}.
\]
The specialisation
\[
A_\alpha(1,q) =\left[\!\!\begin{array}{c}n\\\alpha\end{array}\!\!\right]_q = \frac{[n]_q!}{[\alpha_1]_q!\cdot\ldots\cdot[\alpha_k]_q!}
\]
is a $q$-multinomial coefficient. 

For $\alpha=(1^n)\vDash n$, $M_\alpha=[n]$ is a set, $W_\alpha$ is the set of all permutations $\sigma\in S_n$ of $n$, and $\des(\sigma), \maj(\sigma)$ are the classical descent and major indices of a permutation; $A_{(1^n)}(t,q)$ is $q$-refinement of the classical Eulerian polynomial $A_{(1^n)}(t,1)$ enumerating permutations by descent. For $\alpha=(n)\vDash n$, $M_\alpha=\{1^n\}$ and $W_\alpha$ consists of a single element $(1\ldots 1)$, with both descent and major index $0$, so $A_{(n)}(t,q)=1$. Further examples are given below in Examples~\ref{ex21}-\ref{ex22}.

The following combinatorial identity of MacMahon will play an important role. 

\begin{theorem} 
For any composition $\alpha \vDash n$, we have the identity 
\begin{equation} \sum_{r=0}^\infty \left(\prod_{i=1}^k  \left[\begin{array}{c} r+\alpha_i\\ \alpha_i\end{array}\right]_q\right) t^r=  \frac{A_\alpha(t,q)}{\prod_{j=0}^n (1-q^jt)}.
% t^{l-1} 
\label{eq:mmcarlitz} 
\end{equation}
\end{theorem}
\begin{proof} This identity is proved by MacMahon in~\cite[Vol. 2 Sect. IX]{Mac}; for a modern reference, see~\cite[Sect. 3]{Gessel}. The case $\alpha=(1^n)$, corresponding to permutatons, is the so-called Carlitz identity~\cite{Carlitz}.
\end{proof}

\subsection{Standard Young tableaux} \label{subsec:SYT} Given a partition~$\lambda\vdash n$ of~$n$, denote as usual by $\SYT(\lambda)$ the set of standard Young tableaux of shape $\lambda$, the set of fillings $T$ of the Young diagram corresponding to $\lambda$ bijectively by $[n]$ such that the entries in each row and each column are increasing. The descent, respectively major indices of a standard Young tableau $T\in \SYT(\lambda)$, denoted by $\des(T)$ and $\maj(T)$, are defined as the number, respectively sum, of all $1\leq i\leq n-1$ such that~$i + 1$ appears in a lower row in~$T$ than~$i$.

\subsection{Refined Ehrhart theory}\label{sec_qEhr}

Let $\Sigma\subset\R^n$ be an $n$-dimensional convex lattice polytope, the convex hull of finitely many lattice points in $\Z^n\subset\R^n$. 
For a non-negative integer~$r$, let $r\Sigma\subset\R^n$ denote the $r$-fold dilation of $\Sigma$. Fix an $n$-tuple of integer weights ${\bf a}=(a_1, \ldots, a_n)$. Following Chapoton~\cite{Chapoton}, for $k\in\Z$ consider the quantity
\[ L^\ba_\Sigma(r,k) = \#\{ v \in r\Sigma \cap \Z^n : \bv\cdot {\ba} =k \}.\]
Geometrically, this corresponds to slicing the polytope~$\Sigma$ with hyperplanes perpendicular to the direction~$\ba$, and counting lattice points in the slices. Let
\[ L^\ba_{\Sigma,r}(q) = \sum_{k\in\Z} L^\ba_\Sigma(r,k)q^k,\]
a Laurent polynomial in the variable $q$, and consider the series
\[ E^\ba_{\Sigma}(t,q) = \sum_{r=0}^\infty  t^r L^\ba_{\Sigma,r}(q)  .
\]
Setting $q=1$, $L^\ba_{\Sigma,r}(1)$ and $E^\ba_{\Sigma}(t,1)$ are the Ehrhart polynomial and Ehrhart series of $\Sigma$, respectively. 
%We will hence call $E^\ba_{\Sigma}(t,q)$ the $q$-refined Ehrhart series of $(\Sigma,\ba)$. 

Consider the polytope algebra of $\Sigma$, the $\C$-algebra
\[
\pa_\Sigma =  \C[x_{\bv,r} \mid r\geq 0, \bv \in r\Sigma \cap \Z^d]/\langle  x_{\bv,r}\cdot x_{\bv',r'} - x_{\bv+\bv', r+r'}\rangle.\]
This is a naturally a bigraded algebra, with the generator $x_{\bv,r}$ having bidegree $(r,\bv\cdot \ba)$. Using only the first grading, $\pa_\Sigma$ is non-negatively graded, and $\Proj A_\Sigma$ is the projective toric variety defined by the polytope $\Sigma$. By definition, 
\[\dim_{t,q} \pa_\Sigma =E^\ba_{\Sigma}(t,q).\]

A simple example of this setup  will be needed later below; compare~\cite{AS}. Let $\Sigma=\Delta_n\subset\R^n$  be the basic rectangular simplex in $n$-space, with vertices at the origin and the endpoints of the standard orthonormal basis vectors. Let $\ba=(1,\ldots, n)\in\Z^n$. The corresponding polytope algebra is  $\pa_{\Delta_n}\cong\C[\P^n]\cong \C[z_0,\dots, z_n]$, bigraded by $\deg(z_j)=(1,j)$.
\begin{lemma} \label{lem:Pn}
Under this bigrading, the $\C$-algebra $\C[\P^n]$ has bigraded Hilbert series
\begin{equation}
E^\ba_{\Delta_n}(t,q) = \dim_{t,q} \C[\P^n] =\prod_{j=0}^n (1-tq^j)^{-1} =  \sum_{r\geq 0}  t^r \left[\begin{array}{c} r+n\\ n\end{array}\right]_q.
\label{qseries}
\end{equation}
Hence the $q$-count of lattice points in the $r$-th dilation is
\[ L^\ba_{\Delta_n,r}(q) = \left[\begin{array}{c} r+n\\ n\end{array}\right]_q.\]
\end{lemma}
\begin{proof} As $\C[\P^n]$ is freely generated by the variables $z_i$, the second equality is clear. The third equality is the negative $q$-binomial theorem. 
\end{proof} 

\subsection{Families of algebras} One of the key themes of this paper is the appearance of families of algebras over the affine plane. Such a family is a flat homomorphism of finitely generated (often finite-dimensional) $\C$-algebras $\C[s_0, s_1]\rightarrow \mathcal{A}$, corresponding to a flat family of affine varieties
\[\pi\colon \Spec\mathcal{A} \to \Spec\C[s_0, s_1]\cong{\mathbb A}^2. 
\]
For some fixed $(t_0, t_1)\in\A^2$, denote 
\[
\mathcal{A}_{(t_0, t_1)} = \mathcal{A}\otimes_{\C,(t_0, t_1)} \C[s_0, s_1]
\]
the fibre algebra of this family, where the tensor product is defined via the algebra homomorphism $\C[s_0, s_1]\to\C$ being evaluation at $(t_0,t_1)$. Then clearly
\[\mathcal{A}_{(t_0, t_1)}  \cong \C[\pi^{-1}(t_0, t_1)],
\]
the coordinate algebra of the corresponding geometric fibre of $\pi$. 

%As explained in, these formulae have a simple interpretation in polytope geometry. Namely, The generating function of $q$-counts of lattice points in dilations of the simplex $\Delta_n$, where each lattice point $(l_1,\ldots, l_n)\in \Z^n$ has $q$-weight $\sum_{i=1}^ni\cdot l_i$, gives the $q$-Ehrhart series~\eqref{qseries}, the bigraded Hilbert series of the polytope ring. 
%The individual coefficient $\left[\begin{array}{c} r+n\\ n\end{array}\right]_q$ corresponds to the $q$-count of lattice points in the $r$-th dilation $r\cdot\Delta_n$, the variable $t$ being responsible for the first grading. 

\section{The projective coinvariant algebra and its deformations}
\subsection{The basic definition}\label{section:basic_def}  Consider the bigraded $\C$-algebra
\[ \pa_n = \C[x_{\scriptscriptstyle I}\colon I\subseteq[n]]/ J_n, 
\]
where the generator $x_{\scriptscriptstyle I}$ has bidegree $(1,|I|)$ and $J_n$ is the bigraded toric ideal 
\[
J_n=\left\langle x_{\scriptscriptstyle I}x_{\scriptscriptstyle J}-x_{\scriptscriptstyle I\cap J}x_{\scriptscriptstyle I\cup J}\colon I,J\subseteq[n]\right\rangle. 
\]
The bigrading on $\pa_n$ can equivalently be thought of as the action of the two-dimensional torus $(\C^*)^2$ on $\pa_n$ defined by 
\[
 (t_0, t_1) \circ x_{\scriptscriptstyle I} =  t_0 t_1^{|I|} x_{\scriptscriptstyle I}. 
\]
In the language of Section~\ref{sec_qEhr}, the algebra $\pa_n$ is the polytope algebra of the $n$-dimensional unit hypercube, with generators $x_{\scriptscriptstyle I}$ corresponding to 
vertices of the hypercube. The bigrading corresponds to considering $q$-Ehrhart theory with $\ba=(1,\ldots, 1)$. As it is well known, when considered graded by the first component only, $\Proj\pa_n$ is the image of the Segre embedding
\[s_n\colon (\P^1)^n\rightarrow \P^{2^n-1}.\]
Indeed, the Segre embedding can be given in coordinates by 
\[\begin{array}{rcccc}
s_n& \colon &  (\P^1)^n & \rightarrow &  \P^{2^n-1}\\
&& \left([a_1:b_1], \ldots, [a_n:b_n] \right)& \mapsto & \left[\ldots\colon \prod_{i\in I} a_i\prod_{j\not\in I}b_j : \ldots\right]
\end{array}\]
The ideal $J_n$ records all relations between the image coordinates. So $\pa_n$ can be thought of as the projective coordinate ring of the image of the Segre embedding~$s_n$, equipped with an extra grading. 

We will occasionally find it useful to consider the following alternative model of this algebra. Consider the free bigraded commutative algebra
$\C[\zeta,\xi_1, \ldots, \xi_n]$ with $\deg(\zeta) =(1,0)$ and $\deg(\xi_i)=(0,1)$. Then it is well known that $T_n$ is isomorphic to a bigraded subalgebra of 
$\C[\zeta,\xi_1, \ldots, \xi_n]$, the image of the embedding
\begin{equation} \begin{array}{rcl} T_n & \hookrightarrow & \C[\zeta,\xi_1, \ldots, \xi_n]\\
x_I&\mapsto& \zeta\cdot \prod_{i\in I} \xi_i.
\end{array}\label{explicitembedding}
\end{equation}
Indeed the relations $x_{\scriptscriptstyle I}x_{\scriptscriptstyle J}=x_{\scriptscriptstyle I\cap J}x_{\scriptscriptstyle I\cup J}$ for $I,J\subseteq[n]$ are automatically satisfied under this map. 

For $0\leq k\leq n$, denote
\[\te_k= \sum_{\substack{I\subseteq[n]\\|I|=k}} x_{\scriptscriptstyle I}\in \pa_n. 
\]
Note $\te_0=x_{\scriptscriptstyle\emptyset}$, $\te_n = x_{\scriptscriptstyle [n]}$, and that the element $\te_k\in\pa_n$ is homogeneous of bidegree $(1,k)$.

The following results were proved in~\cite[Prop.2.7]{AS}; compare also~\cite[Section 4]{ORS}. (The statements also appear in~\cite{BO}, but with an incomplete argument for the first statement.)
\begin{proposition} \label{prop_reg}
The elements $\te_0, \ldots, \te_n$ form a regular sequence of maximal length in the algebra~$\pa_n$. The quotient
\[P_n = \pa_n / \left\langle \te_0, \ldots, \te_n\right\rangle\]
is a finite-dimensional bigraded Gorenstein $\C$-algebra, with bigraded Hilbert series
\[
\dim_{t,q} P_n = A_{(1^n)}(t,q)=\sum_{\sigma\in S_n} t^{\mathrm{des}(\sigma)}q^{\mathrm{maj}(\sigma)}.
\]
\end{proposition}

Call the bigraded algebra $P_n$, first introduced independently in~\cite{BO, AS} (and, without its bigrading, in~\cite{ORS}), the {\em projective coinvariant algebra}; the name will be justified by Theorem~\ref{thm_main_def} below. 
%Denote $q_n\colon T_n\to P_n$ the quotient map. 

%\begin{remark} The defining relations of the algebra $\pa_n$ are closely related to those of the Stanley--Reisner ring of the poset. Will explain the difference if asked.\end{remark}

\subsection{A family of deformations}\label{sec:def}

Introduce the $\C[s_0, s_1]$-algebras
\[ \CT_n = \pa_n[s_0, s_1]\]
and
\[ \CP_n = \CT_n/\langle x_{\scriptscriptstyle\emptyset} - s_0, \te_1, \ldots, \te_{n-1}, x_{\scriptscriptstyle [n]}-s_1\rangle. 
\]
We can think geometrically of the map \[\pi_n\colon \Spec\CP_n\to \Spec\C[s_0, s_1]=\A^2\] induced by the inclusion $\C[s_0, s_1]\rightarrow \CP_n$ as a giving a two-parameter family of (spectra of) algebras over the affine plane $\A^2$, a subfamily of the trivial family 
\[\tau_n\colon \Spec\CT_n\to \Spec\C[s_0, s_1]=\A^2.\] 

Consider the action of the torus $(\C^*)^2$ on $\C[s_0, s_1]$ given by 
\[ (t_0, t_1) \circ (s_0, s_1) = (t_0 s_0, t_0 t_1^n s_1). 
\]
This can be extended to a compatible action of $(\C^*)^2$ on $\CT_n$ and~$\CP_n$ over $\C[s_0, s_1]$. 
The symmetric group $S_n$ also acts on~$\CT_n$ and thus on $\CP_n$ by combining its standard action on $[n]$ inducing an action on all subsets with the trivial action on $\C[s_0, s_1]$.

\begin{theorem}\label{thm_main_def}\begin{enumerate}
\item[(i)] The inclusion $\C[s_0, s_1]\rightarrow \CP_n$ is a flat map of algebras, equivariant with respect to the action of the torus $(\C^*)^2$, as well as the symmetric group $S_n$. 
\item[(ii)] The central fibre $(\CP_n)_{(0,0)}$ is bigraded $S_n$-isomorphic to the projective coinvariant algebra $P_n$. 
\item[(iii)] Over points $(s_0, 0)$ and $(0,s_1)$ with $s_i\neq 0$, the fibres $(\CP_n)_{(s_0,0)}$, $(\CP_n)_{(0,s_1)}$ are graded $S_n$-isomorphic to the classical coinvariant algebra $R_n$.
\item[(iv)] Over points $(s_0, s_1)$ with $s_0s_1\neq 0$, the fibre $(\CP_n)_{(s_0,s_1)}$ is an (ungraded) finite-dimensional algebra of dimension $n!$, $S_n$-isomorphic to the group algebra $\C S_n$. 
\end{enumerate}
\end{theorem}
\begin{proof}

The central fibre~$(\CP_n)_{(0,0)}$ is by definition isomorphic to the projective coinvariant algebra $P_n$. On the other hand, being a regular sequence is an open property, so the corresponding elements give regular sequences in fibres near the origin, and hence, using the torus action, over the whole affine plane. 

Let us consider the fibre $(\CP_n)_{(s_0,0)}$ over a point $(s_0, 0)$ with $s_0\neq 0$. Re-scaling using the first component of the torus action, we can assume that $s_0=1$. In this algebra $(\CP_n)_{(1,0)}$, the image of the element $x_{\scriptscriptstyle \emptyset}$ gets identified with the unit. Define a $\C$-algebra homomorphism 
\[\psi_n\colon \C[u_1, \ldots, u_n]\to (\CP_n)_{(1,0)}\]
by $\psi_n(u_i)=(x_{\scriptscriptstyle\{i\}})_{(1,0)}$. Then for $i\not\in I$, using the relations $x_{\scriptscriptstyle\emptyset} x_{\scriptscriptstyle I\cup \{i\}} = x_{\scriptscriptstyle\{i\}} x_{\scriptscriptstyle I}$ inductively, we deduce 
\[\psi_n\left(\prod_{i\in I} u_i\right) = (x_{\scriptscriptstyle I})_{(1,0)}  \in (\CP_n)_{(1,0)}.\]
In particular, if $e_k(u_1,\ldots, u_n)$ denotes the elementary symmetric functions of the variables $u_1, \ldots, u_n$, we get
\[\psi_n\left(e_k(u_1,\ldots, u_n)\right) = \sum_{\substack{I\subseteq[n]\\|I|=k}} (x_{\scriptscriptstyle I})_{_{(1,0)}} =  0\in (\CP_n)_{(1,0)}.\]
It is easy to see that these elements generate the kernel of $\psi_n$. We deduce the isomorphism 
\[ (\CP_n)_{(1,0)} \cong  \C[u_1\,\ldots, u_n]/\langle e_1(u_i), \ldots, e_n(u_i)\rangle = \C[u_1\,\ldots, u_n]/\langle \C[u_1\,\ldots, u_n]^{S_n}_+\rangle
\]
with the classical coinvariant algebra $R_n$. 
The subtorus $\{t_0=1\}\subset (\C^*)^2$ fixes the point $(1,0)\in\A^2$, and acts on the fibre algebra $(\CP_n)_{(1,0)}$, equipping it  with a grading agreeing with the standard grading on the coinvariant algebra. The remaining statements in (iii) are then well known. 

The fibre over a point $(0, s_1)$ is isomorphic to the fibre over a point $(s_0, 0)$ using the ``external'' automorphism of the whole setup which interchanges the coordinate $x_{\scriptscriptstyle I}$ with the coordinate $x_{\scriptscriptstyle [n]\setminus I}$, and the deformation coordinates $s_0$ and $s_1$ (the central reflection of the hypercube). 

Finally for general fibres $(\CP_n)_{(s_0,s_1)}$ with $s_0s_1\neq 0$, we can re-scale to $s_0=s_1=1$. Setting $v_i=(x_{\{i\}})_{(1,1)}\in(\CP_n)_{(1,1)}$, much of the argument of the previous paragraphs goes through, except for the last step where the final relation gets deformed by the parameter $s_1=1$. We get
\[ (\CP_n)_{(1,1)}\cong\C[v_1, \ldots, v_n]\Big/\left\langle \sum_{i=1}^n v_i, \sum_{i<j} v_iv_j, \ldots, \prod_{i=1}^n v_i -1\right\rangle. \]
The latter algebra is the coordinate algebra of the free $S_n$-orbit of the point 
\[(1, \eta, \eta^2, \ldots, \eta^{n-1})\in\A^n = \Spec\C[v_1, \ldots, v_n]\]
for a primitive $n$-th root of unity~$\eta$. Hence this (ungraded) algebra is indeed $S_n$-isomorphic to the regular representation of $S_n$.  

Finally the affine map $\pi_n$ is finite with smooth target and fibres of constant length, hence the inclusion $\C[s_0, s_1]\rightarrow \CP_n$ is indeed flat. 
\end{proof}

\begin{remark} The degeneration of the ungraded group algebra $\C S_n$, realised as the coordinate ring of a free orbit of $S_n$ in affine $n$-space, to the (graded) coinvariant algebra $R_n$ is a particularly simple example of the Garsia--Haiman orbit harmonics method~\cite{GH}. 
From this point of view, what Theorem~\ref{thm_main_def} says is that there is a further, bigraded degeneration of geometric and combinatorial interest. 
\end{remark}

% See Kovacs Sanyi conversation: we would alternatively need source being CM. But it is cut out by a regular sequence globally! So CM. 

\subsection{The graded character of the projective coinvariant algebra}

As discussed above, the projective coinvariant algebra $P_n$ is a bigraded version of the regular representation of the symmetric group $S_n$. Its bigraded character is computed by the following formula, the right hand side of this which first appeared in a closely related but slightly different context in~\cite[Prop.3.14]{RSW}.
\begin{theorem} \label{thm:char} We have 
\[ %\label{eq:char}
\cha_{t,q} P_n  =   \sum_{\lambda\vdash n} \left(\sum_{T\in\SYT(\lambda)} t^{\mathrm{des}(T)}q^{\mathrm{maj}(T)}\right)s_\lambda.
\]
%Here as usual $\lambda\vdash n$ are partitions of~$n$, $s_\lambda$ the corresponding Schur function (irreducible character), and $\SYT(\lambda)$ denotes the set of standard Young tableaux of shape $\lambda$, equipped with standard descent and major indices. 
\end{theorem}
\begin{proof}\!\!\footnote{I would like to thank Anton Mellit for explaining details of the following calculation.} 
Consider the decomposition $T_n = \bigoplus_{r\geq 0} (T_n)_r$, where we use the first grading.  Each graded piece $(T_n)_r$ is still graded by the second grading; use the notation $\dim_q$ for the corresponding graded dimensions. In the alternative model of the algebra $(T_n)$ introduced in~\eqref{explicitembedding}, 
we have an $S_n$-equivariant isomorphism
\[(T_n)_r \cong \C[\xi_1, \ldots, \xi_n]_{\leq r},\]
where the subscript $()_{\leq r}$ means terms of degree less than equal to $r$ in {\em each} of the indicated variables, and the action of $S_n$ is the standard one on the right hand side. 

As we will be using standard results in $S_n$-invariant theory, in this proof we will index Young subgroups of $S_n$ by partitions. So for a partition $\lambda = (\lambda_1\geq \ldots \geq \lambda_k\geq 1)\vdash n$, denote 
$S_\lambda=S_{\lambda_1}\times \ldots \times S_{\lambda_k}$ the standard Young subgroup of $S_n$. We then have 
\[ \cha_{q} (T_n)_r  =  \displaystyle\sum_{\lambda\vdash n} \dim_q((T_n)_r)^{S_\lambda} \cdot  m_\lambda.
\]
Further, clearly
\[(\C[\xi_1, \ldots, \xi_n]_{\leq r})^{S_\lambda} \cong \C[\xi_1,\ldots, \xi_{\lambda_1}]_{\leq r}^{S_{\lambda_1}} \otimes \ldots \otimes \C[\xi_{\lambda_1+\ldots+\lambda_{k-1}+1},\ldots, \xi_{n}]_{\leq r}^{S_{\lambda_k}}.
\]
For a single set of variables, say the first one, we have 
\[\begin{array}{rcl}\dim_q \C[\xi_1,\ldots, \xi_{\lambda_1}]_{\leq r}^{S_{\lambda_1}} & = & \displaystyle\sum_{0\leq a_1\leq \ldots\leq a_k\leq r} q^{a_1+\ldots + a_k}\\ \\
& = & h_{\lambda_1}(1,q,\ldots, q^r).
\end{array}
\]
So we deduce
\[\begin{array}{rcl}
\dim_q((T_n)_r)^{S_\lambda} & = & \displaystyle\prod_{i=1}^k h_{\lambda_i}(1,q, \ldots, q^r) \\
& = & h_\lambda[[r+1]_q],
\end{array}
\]
and hence
%\[\begin{array}{rcl} \cha_{q} (T_n)_r & = & \displaystyle\sum_{\lambda\vdash n} \dim_q((T_n)_r)^{S_\lambda} \cdot  m_\lambda \\
%& = & \displaystyle\sum_{\lambda\vdash n} h_\lambda[1+q+\ldots +q^r] \cdot m_\lambda.
%\end{array}\]
\[ \cha_{q} (T_n)_r = \displaystyle\sum_{\lambda\vdash n} h_\lambda[[r+1]_q] \cdot m_\lambda.
\]
On the other hand, using ${\bf z}$ as a new set of dummy variables for plethystic substitution, using standard identities we get
\[ \begin{array}{rcl}
\displaystyle\sum_{\lambda\vdash n} h_\lambda[[r+1]_q] \cdot m_\lambda({\bf z})  & = & h_n\left[ [r+1]_q {\bf z} \right]\\
& = & \displaystyle\sum_{\lambda\vdash n}s_\lambda \left([r+1]_q \right) s_\lambda({\bf z}). 
\end{array}\]
We obtain
\[\cha_{q} (T_n)_r =  \sum_{\lambda\vdash n}s_\lambda([r+1]_q)s_\lambda\]
and thus, summing over all $r$,
\[\cha_{t,q} T_n =  \sum_{r\geq 0} t^r \sum_{\lambda\vdash n}s_\lambda([r+1]_q)s_\lambda. \]
For the finite-dimensional quotient $P_n$, we hence get 
\[\cha_{t,q} P_n = \prod_{j=0}^n(1-t q^j)\left(\sum_{r\geq 0} t^r \sum_{\lambda\vdash n}s_\lambda([r+1]_q)s_\lambda \right). 
\] 
Standard algebraic manipulations as in~\cite[Proof of Prop.3.14]{RSW} conclude the proof. 
\end{proof}
Note that this result fits well with the specialization picture of Theorem~\ref{thm_main_def}, giving a bigraded refinement of the classical Lusztig--Stanley formula~\eqref{eq:StanLusz} for the (singly) graded character of the coinvariant algebra. 

\begin{corollary}\label{cor:Pn} We have \[\pa_n^{S_n} \cong \C[\te_0, \ldots, \te_n].\]
\end{corollary}
\begin{proof} The elements $\te_0, \ldots, \te_n$ are $S_n$-invariant. They also form a regular 
sequence in $\pa_n$, so they generate a free polynomial subalgebra. On the other hand, 
the quotient 
\[P_n = \pa_n / \left\langle \te_0, \ldots, \te_n\right\rangle\]
has only a one-dimensional $S_n$-invariant subspace by Theorem~\ref{thm:char}, the subspace of constants. Hence $\C[\te_0, \ldots, \te_n]$ contains all $S_n$-invariant elements in $\pa_n$. 
\end{proof}

\subsection{Connection to the diagonal coinvariant algebra}
In this section, a ``nontrivially trivial'' connection between the projective coinvariant algebra $P_n$ and the Garsia--Haiman diagonal coinvariant algebra $D_n$ will be discussed. Recall that the diagonal coinvariant algebra is the bigraded algebra
\[ D_n = \C[a_1, \ldots, a_n, b_1, \ldots, b_n] /\langle \C[a_1, \ldots, a_n, b_1, \ldots, b_n]^{S_n}_+\rangle, 
\]
where the symmetric group $S_n$ acts on $\C[a_1, \ldots, a_n, b_1, \ldots, b_n]$ by permuting the two sets of variables simultaneously. The bigrading is defined by $\deg a_i=(1,0)$ and $\deg b_j=(0,1)$. The algebra $D_n$ has the structure of a bigraded $S_n$-module of dimension $(n+1)^{n-1}$~\cite{H}. 

The following statement is formulated in a somewhat tongue-in-cheek way, but the point is that (i) is easy and natural, whereas (ii) is a nontrivial statement. 

\begin{proposition} \label{prop:PnDn}
\begin{enumerate}
\item[(i)] There is a natural $S_n$-equivariant bigraded algebra homomorphism \[\varphi_n\colon P_n\to D_n.\] 
\item[(ii)] The homomorphism $\varphi_n$ is trivial, factoring into the composition $P_n\to P_n/(P_n)_+\cong\C\hookrightarrow D_n$. 
\end{enumerate}
\end{proposition}
\begin{proof} The homomorphism~$\varphi_n$ was essentially defined already earlier in the paper: for $I\subseteq [n]$, let  
\[\tilde\varphi_n\colon x_{\scriptscriptstyle I} \mapsto \prod_{i\in I} a_i\prod_{j\not\in I}b_j. 
\]
The relations in $I_n$ are satisfied in~$\C[a_1, \ldots, a_n, b_1, \ldots, b_n]$ (this is just the Segre embedding again!), so we get an algebra homomorphism $\widetilde{\varphi}_n\colon \pa_n\to \C[a_1, \ldots, a_n, b_1, \ldots, b_n]$, that is bigraded under a re-scaling of the bigrading on $T_n$. The elements $\te_i$ are mapped to $S_n$-invariant elements in the target, and hence $\widetilde{\varphi}_n$ descends to an $S_n$-equivariant bigraded morphism $\varphi_n\colon P_n\to D_n$, proving (i). 

The nontrivial part is (ii), the statement that this natural homomorphism is actually trivial: elements of the form $\prod_{i\in I} a_i\prod_{j\not\in I}b_j$ for $I\subseteq [n]$ are all zero in the diagonal coinvariant algebra $D_n$. I~see no easy reason for this fact, observed first in computer calculations\footnote{I thank Erik Carlsson for explaining this proof.}. 

We may assume for simplicity that $I=\{1,\ldots,k\}$. The main claim is that if we regard $D_n$ as a module over $\C[a_1,\ldots,a_n]$, then the submodule spanned by $b_{k+1}\cdot\ldots\cdot b_n$ is isomorphic to the submodule of the usual coinvariant algebra $R_n$ on the generators $\{a_1, \ldots, a_n\}$ spanned by $a_{k+1}\cdot\ldots\cdot a_{n}$. The vanishing statement then follows, because if we multiply that submodule by $a_1\cdot\ldots a_k$, we clearly get $0$.

The reason for the claim is that the product $b_{k+1}\cdot\ldots\cdot b_{n}$ is the monomial of degree $n-k$ that is lowest in the Garsia--Stanton descent order. The statement then follows from~\cite[Thm. B, c)]{CO}, applied at $F_a=F_{(0,\ldots,0,1,\ldots,1)}$. 
\end{proof}

\begin{remark} \label{rem:PnDn} The three finite-dimensional algebras $R_n, P_n, D_n$ are connected by the schematic diagram
\[ R_n \rightsquigarrow P_n \stackrel{\varphi_n}\longrightarrow D_n, \]
where $\rightsquigarrow$ denotes specialisation. These algebras arise from the $S_n$-equivariant coordinate geometry, respectively, of $(\A^1)^n$, $(\P^1)^n$ and $(\A^2)^n$. So it makes sense that there are connections between them. However, the triviality of the natural homomorphism $\varphi_n$ remains mysterious to me. 
\end{remark}

\section{Young subgroups and general Segre embeddings}

\subsection{The bigraded coordinate ring of a general Segre embedding}\label{sec:Talpha} Consider a composition $\alpha \vDash n$, and
let $\P_\alpha = \P^{\alpha_1}\times \ldots \times \P^{\alpha_k}$ with 
$\P^{\alpha_i}=\Proj \C\left[z_0^{(i)},\ldots, z_{\alpha_i}^{(i)}\right]$.
Let also 
\[
N_\alpha = |{\mathcal P}(M_\alpha)| - 1 = \prod_{i=1}^k(\alpha_i+1) -1
\]
and 
%\[\P^{N_\alpha} = \Proj \C\left[y_{j_1\ldots j_k} \colon {0\leq j_i\leq \alpha_i}\right]. \]
\[\P^{N_\alpha} = \Proj \C\left[y_{I} \colon {I\subset M_\alpha}\right],\]
the variables being indexed by the set of subsets ${\mathcal P}(M_\alpha)$ of the multiset $M_\alpha$.
The Segre embedding
\[ s_\alpha\colon \P_\alpha = \P^{\alpha_1}\times \ldots \times \P^{\alpha_k}\rightarrow \P^{N_\alpha}
\]
is defined by the dual map on coordinate rings given by 
\[
y_{\scriptscriptstyle I}\mapsto z_{j_1}^{(1)}z_{j_2}^{(2)}\ldots z_{j_k}^{(k)} \mbox { for } I=\{1^{j_1}\ldots k^{j_k}\}\subset M_\alpha \mbox{ with } 0\leq j_i\leq \alpha_i. 
\]
%\[ \left([z_0^{(1)}:\ldots : z_{\alpha_1}^{(1)}], \ldots,[z_0^{(k)}:\ldots : z_{\alpha_k}^{(k)}] \right)\mapsto \left[\ldots\colon \left(z_{j_1}^{(1)}z_{j_2}^{(2)}\ldots z_{j_k}^{(k)}\right) : \ldots\right]_{0\leq j_i\leq \alpha_i}.\]
Let 
%\[\pa_\alpha  = \C\left[y_{j_1\ldots j_k} \colon {0\leq j_i\leq \alpha_i}\right]/\langle y_{j_1\ldots j_i\ldots j_k} y_{j'_1\ldots j'_i\ldots j'_k} - y_{j_1\ldots j'_i \ldots j_k} y_{j'_1\ldots j_i\ldots j'_k}\colon {0\leq j_i, j_i'\leq \alpha_i}\rangle\]
\[\pa_\alpha  = \C\left[y_I \colon I\subset M_\alpha\right]/\langle y_{\scriptscriptstyle I_1} y_{\scriptscriptstyle I_2} - y_{\scriptscriptstyle I_3} y_{\scriptscriptstyle I_4}\colon (I_1,I_2)\sim (I_3,I_4)\rangle\]
be the projective coordinate algebra of the image of the Segre embedding $s_\alpha$. Here the equivalence relation 
$(I_1,I_2)\sim (I_3,I_4)$ on pairs of multisets with $I_m=(1^{j_{1,m}}\ldots k^{j_{k,m}})\subset M_\alpha$ is defined as follows: for all $1\leq i\leq k$, we want $\{ j_{i,1}, j_{i,2}\} = \{j_{i,3}, j_{i,4} \}$ in the sense of multisets. Compare~\cite[Section 4]{ORS}.

The algebra $T_\alpha$ can be given a bigrading as follows: 
%the Segre coordinate $y_{j_1,\ldots, j_k}$ in the target $\P^{N_\alpha}$ has bidegree $(1,\sum_{i=1}^k j_i)$. 
the Segre coordinate $y_{\scriptscriptstyle I}$ in the target $\P^{N_\alpha}$ has bidegree $(1,|I|)$. 
It is immediately seen that this bigrading is compatible with the defining ideal of the image of Segre embedding. 

\begin{remark} A perhaps more familiar way to index the Serge variables would be
\[\P^{N_\alpha} = \Proj \C\left[y_{j_1\ldots j_k} \colon {0\leq j_i\leq \alpha_i}\right], \]
with the Segre embedding given by 
\[ \left([z_0^{(1)}:\ldots : z_{\alpha_1}^{(1)}], \ldots,[z_0^{(k)}:\ldots : z_{\alpha_k}^{(k)}] \right)\mapsto \left[\ldots\colon \left(z_{j_1}^{(1)}z_{j_2}^{(2)}\ldots z_{j_k}^{(k)}\right) : \ldots\right]_{0\leq j_i\leq \alpha_i}.\]
This is easily seen to be equivalent to the above formalism, which fits better with the spirit of this paper. 
\end{remark}

Once again, it will be useful to have an alternative model of this algebra, similar to the one given in~\eqref{explicitembedding} for $T_n$. Given a composition $\alpha =(\alpha_1, \ldots, \alpha_k)\vDash n$, 
consider the free bigraded commutative algebra
$\C[\zeta,\xi_{i,l} \colon 1\leq i \leq k, 1\leq l \leq \alpha_i]$ with $\deg(\zeta) =(1,0)$ and $\deg(\xi_{i,l})=(0,1)$. Then once again, the bigraded algebra $T_\alpha$ is isomorphic to a bigraded subalgebra of 
$\C[\zeta,\xi_{i,l}]$, the image of the embedding
\begin{equation} \begin{array}{rcl} T_\alpha & \hookrightarrow & \C[\zeta, \xi_{i,l}]\\
y_{\scriptscriptstyle I}&\mapsto& \zeta\cdot \displaystyle\prod_{i=1}^k \xi_{i, j_i}  \mbox { for } I=(1^{j_1}\ldots k^{j_k})\subset M_\alpha,
\end{array}\label{generalexplicitembedding}
\end{equation}
where we set $\xi_{i,0}=1$ if $j_i=0$. 
Indeed the relations $y_{\scriptscriptstyle I_1} y_{\scriptscriptstyle I_2} = y_{\scriptscriptstyle I_3} y_{\scriptscriptstyle I_4}$ for $(I_1,I_2)\sim (I_3,I_4)$ are again automatically satisfied under this map.
\begin{example}\label{ex21_ring} Consider the first non-trivial case $\alpha=(2,1)\vDash 3$, corresponding to
$\P_{(2,1)}=\P^2\times\P^1$. Then 
$M_{(2,1)}=\{1,1,2\}$. Then we get the familiar $\C$-algebra
\[\pa_{(2,1)}  = \C[y_\emptyset,y_{\{1\}}, y_{\{1,1\}}, y_{\{2\}}, y_{\{1,2\}}, y_{\{1,1,2\}}]\Big/ \left\langle\mathrm{rank}\left|\begin{array}{ccc} y_\emptyset & y_{\{1\}} & y_{\{1,1\}} \\ y_{\{2\}} & y_{\{1,2\}}& y_{\{1,1,2\}} \end{array}\right|\leq 1\right\rangle,\]
the well-known coordinate algebra of the the standard Segre embedding $\P^2\times\P^1\hookrightarrow\P^5$. One relation is $y_{\{1\}}y_{\{1,1,2\}} = y_{\{1,1\}}y_{\{1,2\}}$, 
corresponding to the equivalent pairs of multi-subsets $(\{1\}, \{1^2, 2\})\sim (\{1^2\}, \{1,2\})$ of $M_\alpha=\{1^2,2\}$.
\end{example}

\subsection{The bigraded Hilbert series}
\begin{theorem} For any composition $\alpha \vDash n$, the bigraded Hilbert series of the algebra $\pa_\alpha$ is given by
\[
\dim_{t,q}  \pa_\alpha=  \frac{A_\alpha(t,q)}{\prod_{j=0}^n (1-tq^j)},
\]
the numerator being the $(\des, \maj)$-generating function of words $W_\alpha$ in the multiset $M_\alpha$. \label{thm:bigradedPa}
\end{theorem}
\begin{proof} Use refined Ehrhart theory. The projective
coordinate ring of the product space $\P_\alpha$ is the polytope algebra $\pa_{\Delta_\alpha}$ of the product
polytope $\Delta_\alpha=\Delta_{\alpha_1} \times \ldots \times \Delta_{\alpha_k}\subset\R^n$. We 
%are interested in the generating function of $q$-counts of lattice points of dilates of $\Delta_\alpha$, where a lattice point 
%\[\left(l^{(1)}_1,\ldots, l^{(1)}_{\alpha_1}, l_1^{(2)}, \ldots, l_{\alpha_k}^{(k)}\right)\in \Z^n\]
%has $q$-weight
%\[\sum_{i=1}^k \sum_{j=0}^{\alpha_i} j\cdot l_j^{(i)}.\]
have 
\[\dim_{t,q}\pa_\alpha = E^\ba_{\Delta_\alpha}(t,q)\]
for the grading vector \[\ba=(1,\ldots, \alpha_1, 1,\ldots, \alpha_2, \ldots, 1, \ldots, \alpha_k)\in\Z^n.\]
On the other hand, we clearly have
\[r\Delta_\alpha=(r\Delta_{\alpha_1}) \times \ldots \times (r\Delta_{\alpha_k})\subset\R^n,\]
and so, invoking also Lemma~\ref{lem:Pn}, we see that the $q$-count of the 
number of lattice points in $r\Delta_\alpha$ is given by 
\[L^\ba_{\Delta_\alpha,r}(q) =\prod_{i=1}^k  \left[\begin{array}{c} r+\alpha_i\\ \alpha_i\end{array}\right]_q.\]
Hence we get 
\[ 
\dim_{t,q}  \pa_\alpha=\sum_{r=0}^\infty \left(\prod_{i=1}^k  \left[\begin{array}{c} r+\alpha_i\\ \alpha_i\end{array}\right]_q\right) t^r. 
\]
We conclude by invoking MacMahon's identity~\eqref{eq:mmcarlitz}.
\end{proof}
Setting $q=1$, we get the classical Hilbert series of the Segre embedding, with each Segre-variable having degree $1$. We thus deduce the following formula, which was proved earlier by Morales in the unpublished~\cite{morales}, and can also be deduced from~\cite{ORS}.
\begin{corollary} \label{cor:nemcomb}
For any composition $\alpha \vDash n$, the (classical) Hilbert series of the coordinate algebra $\pa_\alpha$ of a general Segre embedding $s_\alpha\colon \P_\alpha \rightarrow \P^{N_\alpha}$ is given by the formula 
\[
\dim_t \pa_\alpha=\frac{\displaystyle\sum_{w\in W_\alpha}t^{\des(w_\alpha)}}{(1-t)^{n+1}}. 
\]
\end{corollary}
The resulting coefficients in the numerator are sometimes called Simon Newcomb numbers. 

\begin{example}\label{ex21} Let us continue Example~\ref{ex21_ring} with $\alpha=(2,1)\vDash 3$ and  
$M_{(2,1)}=\{1,1,2\}$. The set of words in this multiset is \[W_{(2,1)} = \{112, 121, 211\}.\] We get the generating function
\[A_{(2,1)}(t,q) = 1 + tq + tq^2.\]
Hence the bigraded Hilbert function of the embedding $\P^2\times\P^1\hookrightarrow \P^5$ is 
\[
\dim_{t,q}  \pa_{(2,1)}=  \frac{1 + tq + tq^2}{(1-t)(1-tq)(1-tq^2)(1-tq^3)}.
\]
The corresponding singly graded expression is 
\[
\dim_{t}  \pa_{(2,1)}=  \frac{1 + 2t}{(1-t)^4} = \frac{1 - 3 t^2 + 2 t^3}{(1-t)^6},
\]
corresponding to the codimension $2$ determinantal embedding $\P^2\times\P^1\hookrightarrow \P^5$ in Hilbert--Burch format. 
\end{example}

\begin{example}\label{ex22} Consider the case $\alpha=(2,2)\vDash 4$. Then 
$M_{(2,2)}=\{1,1,2,2\}$ and \[W_{(2,2)} = \{1122, 1212, 1221, 2112, 2121, 2211\}.\] We get the bigraded Hilbert function
\[
\dim_{t,q}  \pa_{(2,2)}=  \frac{1 + t(q+2q^2+q^3) + t^2q^4}{(1-t)(1-tq)(1-tq^2)(1-tq^3)(1-tq^4)}
\]
with palindromic numerator. The singly graded expression is 
\[
\dim_{t}  \pa_{(2,2)}=  \frac{1 + 4t+t^2}{(1-t)^5} = \frac{1 - 9 t^2 + 16 t^3 - 9t^4+t^6}{(1-t)^9},
\]
corresponding to the standard codimension $4$ embedding $\P^2\times\P^2\hookrightarrow \P^8$. 
\end{example}
We will need a generalisation of Proposition~\ref{prop_reg}. For $0\leq l\leq n$, denote
%\[\tf_l= \sum_{\substack{j_1,\ldots, j_k \\ \sum j_i=l}} y_{j_1\ldots j_k}\in \pa_\alpha. \]
\[\tf_l= \sum_{\substack{I \subset M_\alpha \\ |I|=l}} y_{\scriptscriptstyle I}\in \pa_\alpha. \]
Note that $\tf_0=y_{\scriptscriptstyle\emptyset}$, $\tf_n = y_{\scriptscriptstyle M_\alpha}$, and that the element $\tf_l\in\pa_\alpha$ is homogeneous of bidegree $(1,l)$.

\begin{proposition} For any composition $\alpha \vDash n$, the elements $\tf_0, \ldots, \tf_n$ form a regular sequence of maximal length in the bigraded algebra~$\pa_\alpha$. 
\end{proposition}
\begin{proof} See~\cite[Section 4]{ORS}. The explicit proof of~\cite[Prop.2.7]{AS} also carries over with minimal changes. 
\end{proof}
Define
\[
P_\alpha = T_\alpha/\langle \tf_0,\ldots, \tf_n\rangle, 
\]
a bigraded Artinian Cohen--Macaulay algebra which (without its bigrading) was first studied in~\cite{ORS}.

\begin{corollary}\label{cor:Palpha} For any composition $\alpha \vDash n$, the bigraded Hilbert series of the algebra $P_\alpha$ is given by
\[
\dim_{t,q}  P_\alpha=  A_\alpha(t,q).
\]
\end{corollary}

\subsection{Invariants under a Young subgroup} 

Given $\alpha \vDash n$, there is a corresponding Young subgroup 
\[S_\alpha=S_{\alpha_1} \times \ldots \times S_{\alpha_k}\subset S_n.\]
The action of the symmetric group $S_n$ on the bigraded algebra $\pa_n$ restricts to an action of $S_\alpha$. 

\begin{lemma}\label{lem:regseq} The $S_n$-invariant elements $\tilde e_0, \ldots, \tilde e_n\in\pa_n$ form a regular sequence in the invariant subalgebra $\pa_n^{S_\alpha}\subset\pa_n$. The quotient 
\[ P_n^{S_\alpha} \cong \pa_n^{S_\alpha}/\langle \tilde e_0, \ldots, \tilde e_n\rangle\subset P_n
\]
is the invariant subalgebra $P_n^{S_\alpha}$ of the projective coinvariant algebra $P_n$. 
\end{lemma}
\begin{proof} Consider the Koszul complex of the elements $\tilde e_0, \ldots, \tilde e_n\in \pa_n$. 
Since taking invariants under a finite group is an exact functor in characteristic zero, taking invariants in the Koszul complex
gives the exact Koszul complex of $\tilde e_0, \ldots, \tilde e_n\in \pa_n^{S_\alpha}$. So these elements form a regular sequence in $\pa_n^{S_\alpha}$ also, and the quotient of $\pa_n^{S_\alpha}$ by the ideal they define is indeed the invariant algebra $P_n^{S_\alpha}\subset P_n$. 
\end{proof}
The action of $S_\alpha$ extends to the families $\CT_n$, $\CP_n$ of deformations of the bigraded algebra $\pa_n$ and its quotient~$P_n$. Denote $\CP_\alpha = \CP_n^{S_\alpha}$.
The maps $\C[s_0, s_1]\rightarrow \CT_n^{S_\alpha}\cong \pa_n^{S_\alpha}[s_0,s_1]$ and $\C[s_0, s_1]\rightarrow \CP_\alpha$ are still flat, the former trivially, the latter by semicontinuity. Once again, we can think geometrically of the maps \[\tau_\alpha\colon \Spec\CT_n^{S_\alpha}\to \Spec\C[s_0, s_1]=\A^2\] 
and\[\pi_\alpha\colon \Spec\CP_\alpha\to \Spec\C[s_0, s_1]=\A^2\] 
as a giving two-parameter families of (spectra of) algebras over the affine plane $\A^2$, equipped an action of the torus $(\C^*)^2$ acting on the base as before. 

\begin{theorem} \label{thm:Palpha_def}
Fix a composition $\alpha \vDash n$.
\begin{enumerate} \label{thm:invariantisPalpha}
\item[(i)] Over the point $(0,0)\in\A^2$, the central fibre $(\CT_n^{S_\alpha})_{(0,0)}$ of the family $\CT_n^{S_\alpha}$  is isomorphic to the bigraded coordinate algebra $\pa_\alpha$ of the Segre embedding $s_\alpha\colon \P_\alpha \rightarrow \P^{N_\alpha}$. 
\item[(ii)] The central fibre $(\CP_\alpha)_{(0,0)}$ of the family $\CP_\alpha$ is isomorphic to the finite-dimensional bigraded algebra $P_\alpha$.
\item[(iii)] Over a point $(s,0)\in\A^2$, the fibre $(\CP_\alpha)_{(s,0)}$ of the family $\CP_\alpha$ is isomorphic to the partial coinvariant algebra $R_\alpha$. 
\item[(iv)] Over a point $(s_0,s_1)\in\A^2$ with $s_0s_1\neq 0$, the fibre $(\CP_\alpha)_{(s_0,s_1)}$ of the family $\CP_\alpha$ is isomorphic to the invariant subalgebra $(\C S_n)^{S_\alpha}<\C S_n$. 
\end{enumerate} 
\end{theorem}
\begin{proof} Taking invariants under a finite group commutes with taking fibres in flat families in characteristic zero. Recalling Theorem~\ref{thm_main_def} as well as $R_n^{S_\alpha}=R_\alpha$, 
(iii) is immediate, and so is (iv). We also get 
$(\CT_n^{S_\alpha})_{(0,0)}\cong T_n^{S_\alpha}$, as well as $(\CP_\alpha)_{(0,0)}\cong P_n^{S_\alpha}$. 

Restrict the family $\CP_\alpha$ to the one-dimensional base $\A^1\times\{0\}\subset\A^2$. 
We obtain a flat family of graded algebras, where the central fibre $(\CP_\alpha)_{(0,0)}\cong P_n^{S_\alpha}$ is graded by its second grading. Up to isomorphism, all other fibres are isomorphic using the torus action to $(\CP_\alpha)_{(1,0)}\cong R_n^{S_\alpha} = R_\alpha$. So by flatness, we get 
\begin{equation}\label{eq:dimq1} \dim_q P_n^{S_\alpha} = \dim_q R_\alpha = \left[\begin{array}{c}n\\\alpha\end{array}\right]_q,
\end{equation}
the $q$-multinomial coefficient. 

The key step of our proof will be to define a bigraded, injective algebra homomorphism 
\[ \chi_\alpha \colon T_\alpha\to T_n^{S_\alpha}\]
that maps the regular sequence $\{\tf_0, \ldots, \tf_n\}\subset T_\alpha$ to the regular sequence 
$\{\tilde e_0, \ldots, \tilde e_n\}\subset T_n^{S_\alpha}$. Assume first that such a map $\chi_\alpha$ has been constructed. Then it descends to an injective map of 
bigraded algebras
\[ \bar\chi_\alpha \colon P_\alpha\to P_n^{S_\alpha}.\]
However, specialising Corollary~\ref{cor:Palpha} to the $q$-grading, we get
\begin{equation}\label{eq:dimq2}
    \dim_q P_\alpha = \left[\begin{array}{c}n\\\alpha\end{array}\right]_q. 
\end{equation} 
Comparing~\eqref{eq:dimq1} and~\eqref{eq:dimq2}, we see that $\bar\chi_\alpha$ has to be
an isomorphism, which implies that $\chi_\alpha$ is also an isomorphism, proving statements (i)--(ii) of the Theorem. 

It therefore remains to construct the homomorphism $\chi_\alpha$, and prove its injectivity.
Recall that 
\[\pa_\alpha  = \C\left[y_{\scriptscriptstyle I} \colon I\subset M_\alpha\right]/\langle y_{\scriptscriptstyle I_1} y_{\scriptscriptstyle I_2} - y_{\scriptscriptstyle I_3} y_{\scriptscriptstyle I_4}\colon (I_1,I_2)\sim (I_3,I_4)\rangle,\]
with the equivalence relation $\sim$ from~\ref{sec:Talpha} above. 
Define
\[\tilde\chi_\alpha \colon \C\left[y_{\scriptscriptstyle I} \colon I\subset M_\alpha\right] \to \C\left[x_K \colon K\subset [n]\right] 
\]
on the generator corresponding to $I=(1^{j_1}\ldots k^{j_k})\subset M_\alpha \mbox{ with } 0\leq j_i\leq \alpha_i$ by 
\[
\tilde\chi_\alpha(y_{\scriptscriptstyle I}) = \sum_{J_1, \ldots, J_k} x_{J_1\sqcup \ldots \sqcup J_k}, 
\]
where the sum runs over all collections $\{{J_1, \ldots, J_k}\}$ with $J_i\subset\{\alpha_1+\ldots +\alpha_{i-1}+1, \ldots, \alpha_1+\ldots +\alpha_i\}$ of size $0\leq j_i\leq \alpha_i$ for each $1\leq i \leq k$. This map respects the bigradings, and maps to elements in $\C\left[x_K \colon K\subset [n]\right]$ that are invariant under the natural action of $S_\alpha$. A routine combinatorial check gives that if $(I_1,I_2)\sim (I_3,I_4)$, then 
\[
\tilde\chi_\alpha(y_{I_1} y_{I_2} - y_{I_3} y_{I_4}) \in \left\langle x_{J_1}x_{J_2}-x_{J_1\cap J_2}x_{J_1\cup J_2}\colon J_1,J_2\subseteq[n]\right\rangle. 
\]
We therefore obtain an induced bigraded algebra homomorphism 
\[ \chi_\alpha \colon T_\alpha\to T_n^{S_\alpha}\subset T_n.\]
By definition, we get $\chi_\alpha(\tf_i)=\te_i$ for all $0\leq i\leq n$.
It remains to prove that the homomorphism $\chi_\alpha$ is injective. 

For this, we will use the models\footnote{I would like to thank Fabi\'an Levi\-c\'an for a key discussion about this point.} of our algebras introduced in~\eqref{explicitembedding}-\eqref{generalexplicitembedding}  In that picture, our algebras were realised as explicit subalgebras of free
commutative algebras  $T_\alpha \subset \C[\zeta, \xi_{i,l}]$ and $T_n\subset \C[\zeta,\xi_1, \ldots, \xi_n]$. 
We claim that in this model, the map $\chi_\alpha \colon T_\alpha\to T_n$ is the restriction of a $\C$-algebra
homomorphism
\[  \omega_\alpha\colon\C[\zeta, \xi_{i,l}]\to \C[\zeta,\xi_1, \ldots, \xi_n]\]
defined as follows. We let $\omega_\alpha(\zeta) = \zeta$, and let
\[\omega_\alpha(\xi_{i,l}) = e_l(\xi_{\alpha_1+\ldots +\alpha_{i-1}+1}, \ldots, \xi_{\alpha_1+\ldots +\alpha_i})
\]
for $1\leq l \leq \alpha_i$, where $e_l$ denotes the $l$-th elementary symmetric polynomial of the given variables. 
Comparing definitions, it is easy to set that indeed \[\chi_\alpha = \omega_\alpha|_{T_\alpha}\colon T_\alpha\to T_n\subset \C[\zeta,\xi_1, \ldots, \xi_n].\]
On the other hand, it is immediate that the map $\omega_\alpha$ between free polynomial algebras is injective. 
Hence so is $\chi_\alpha$, and our proof is complete. 
\end{proof}

\subsection{A further graded character formula}

Given $n=k\cdot m$, consider the composition $\alpha=(m,\ldots, m)\vDash n$. The corresponding Young subgroup is 
$S_\alpha = (S_m)^{\times k}\subset S_n$. This is normalised in~$S_n$ by the symmetric group $S_k$ that permutes $m$-tuples of indices 
simultaneously. The corresponding invariant algebra $P_\alpha = P_n^{S_\alpha}$ thus has a residual bigraded action 
of the symmetric group $S_k$. The formula for its Frobenius character is best expressed once again in plethystic notation. 
\begin{theorem} \cite[Cor.~5,8, Ex.~5.10]{LR} For  $n=k\cdot m$ and $\alpha=(m,\ldots, m)\vDash n$, 
\[\cha_{t,q} P_\alpha =  \prod_{j=0}^n(1-t q^j)\left(\sum_{i\geq 0} t^i \sum_{\mu\vdash k}s_\mu\left(\left[\!\!\begin{array}{c} i+m \\ m \end{array}\!\!\right]_q\right)s_\mu \right)\] 
as bigraded characters of the symmetric group $S_k$. 
\label{thm:fabianetal}
\end{theorem}
It is not known to the authors of~\cite{LR}, and indeed to me, whether there is a formula for this character that is a sum over 
a natural finite combinatorial set in analogy to~\eqref{eq:Pntqchar}, though this seems likely. Compare \cite[Prop.~5.21]{LR} and the surrounding discussion. 

\section{Further directions}

\subsection{Cohomological interpretations}\label{sec:cohom} Let $F_n = {\rm GL}_n(\C)/B$ be the flag variety of all flags $U_\bullet=(U_1\subset\ldots \subset U_{n-1} \subset \C^n)$ with $\dim U_i = i$, where
$B<{\rm GL}_n(\C)$ is the standard Borel subgroup. Then Borel's isomorphism is an isomorphism of graded rings
\[
R_n \cong H^*(F_n, \C)
\]
between the coinvariant algebra and the complex cohomology of the flag variety (the degrees are doubled on the right hand side). One deduces the equality
\[
\dim_{q}H^*(F_n, \C) = [n]_q!
\]
for the Poincar\'e polynomial of the flag variety. 

Further, to a composition $\alpha\Vdash n$ one can associate a parabolic subgroup $P_\alpha <{\rm GL}_n(\C)$, so that the quotient
$F_\alpha={\rm GL}_n(\C)/P_\alpha$ is the variety of partial flags, the space of all filtrations of subspaces $(V_1\subset\ldots \subset V_{k-1} \subset \C^n)$ with $\dim V_i = \alpha_1+\ldots + \alpha_i$. The inclusion $B\leq P_\alpha$ gives rise to a projection map
\[ \pi_\alpha\colon F_n \to F_\alpha\]
that forgets the subspaces in a full flag $U_\bullet$ whose dimensions are not of the form $\alpha_1+\ldots + \alpha_i$ for some~$i$. The pullback map
\[ \pi_\alpha^*\colon H^*(F_\alpha,\C) \to H^*(F_n,\C)\cong R_n \]
is injective, and gives an identification 
\[
H^*(F_\alpha,\C) \cong R_\alpha = R_n^{S_\alpha}
\]
between the cohomology of the partial flag variety and the partial coinvariant algebra $R_\alpha$~\cite[Cor.~5.4, Thm.~5.5]{BGG}. One deduces the equality
\[
\dim_{q}H^*(F_\alpha, \C) = \left[\!\!\begin{array}{c}n\\\alpha\end{array}\!\!\right]_q
\]
for the Poincar\'e polynomial of the partial flag variety. In particular, for the composition $\alpha =(m,n-m)\Vdash n$, we have 
$F_\alpha = {\rm Gr}(m,n)$ the Grassmannian, and the previous equality becomes the well-known
\[
\dim_{q}H^*({\rm Gr}(m,n), \C) = \left[\!\!\begin{array}{c}n\\ m \end{array}\!\!\right]_q. 
\]

A natural question that arises is whether the deformation families $\CP_n, \CP_\alpha$ of the algebras $R_n$, $R_\alpha$ from Theorems~\ref{thm_main_def}-\ref{thm:Palpha_def} have any cohomological interpretations. Let us first write out the deformation (smoothing) of the algebras $R_\alpha$ that one obtains from the family.  

\begin{proposition}\label{prop:cohom}
Given a composition $\alpha=(\alpha_1,\ldots, \alpha_k)\Vdash n$, the one-parameter family of deformations $(\CP_\alpha)_{(1,s)}$ of the
algebra $(\CP_\alpha)_{(1,0)}\cong R_\alpha$ can be described by generators and relations as follows.
Fix variables $u_1^{(1)},\ldots, u_{1}^{(\alpha_1)}, u_2^{(1)}, \ldots, u_{2}^{(\alpha_2)}, \ldots, u_k^{(1)}, \ldots, u_k^{({\alpha_k})}$ with degrees $\deg\left(u_i^{(j)}\right)=j$, and
an extra variable $s$ of degree $\deg s=n$. Then
\[ (\CP_\alpha)_{(1,s)} \cong \C\left[u_1^{(1)}, \ldots,  u_{\alpha_k}^{(k)}\right]\Big/ J_\alpha, 
\] 
where $J_\alpha$ is the ideal in the given polynomial ring generated by degree-homogeneous terms in the polynomial 
\[ \prod_{i=1}^k \left(1+\sum_{j=1}^{\alpha_i} u_i^{(j)}\right)-(1+s).\]
\end{proposition}
\begin{proof} This follows from the proof of Theorem~\ref{thm_main_def}, by unpacking definitions. 
\end{proof} 

Note that for $s=0$, we indeed recover the standard explicit construction of the cohomology of the partial flag variety $F_\alpha$ in terms of Chern classes of universal bundles, with the variable $u_j^{(i)}$ representing the $j$th Chern class of the $i$th universal quotient bundle. 

For two-term compositions $\alpha=(m,n-m)\Vdash n$, corresponding to the partial flag varieties $F_\alpha={\rm Gr}(m,n)$, this deformation 
of the cohomology ring $H^*({\rm Gr}(m,n), \C)$ corresponds to the (small) {\em quantum cohomology} $QH^*({\rm Gr}(m,n), \C)$ with quantum 
parameter $s$, see for example~\cite[Thm.0.1]{ST}. For more general partial flag varieties, and in particular for the flag variety~$F_n$, this deformation does not appear to match\footnote{I would like to thank Konstanze Rietsch for detailed explanations on this point.} any known 
(equivariant) quantum deformation~\cite{GK, CF, Kim}
of the cohomology ring $H^*(F_\alpha, \C)$. 

Two questions thus arise from this discussion. 
\begin{enumerate}
\item Can one give a natural cohomological interpretation to the {\em deformation} (smoothing) of the cohomology ring  $H^*(F_\alpha, \C)$ from Proposition~\ref{prop:cohom} for a partial flag variety $F_\alpha$ parametrising flags with more than one step?
\item Can one give a natural cohomological interpretation to the {\em degeneration} of the cohomology ring $H^*(F_\alpha, \C)$ from Theorem~\ref{thm:Palpha_def}(i) for an arbitrary partial flag variety $F_\alpha$? 
\end{enumerate}

\begin{example} Let us consider $\alpha=(2,1)\vDash 3$ one final time, continuing Examples~\ref{ex21_ring}-\ref{ex21}. We have 
\[T_{(2,1)}  = \C[y_\emptyset,y_{\{1\}}, y_{\{1,1\}}, y_{\{2\}}, y_{\{1,2\}}, y_{\{1,1,2\}}]\Big/ \left\langle\mathrm{rank}\left|\begin{array}{ccc} y_\emptyset & y_{\{1\}} & y_{\{1,1\}} \\ y_{\{2\}} & y_{\{1,2\}}& y_{\{1,1,2\}} \end{array}\right|\leq 1\right\rangle,\]
and the general fibre of the family of deformations of the corresponding $P_\alpha$ is
\[ (\CP_{(2,1)})_{(s_0, s_1)} \cong T_{(2,1)} \Big/ \left\langle y_\emptyset-s_0 , y_{\{1\}} + y_{\{2\}}, y_{\{1,1\}} + y_{\{1,2\}}, y_{\{1,1,2\}}-s_1 \right\rangle.
\]
At $(s_0, s_1)=(1,0)$, using the procedure described in the proof of Theorem~\ref{thm_main_def}, and taking
Young invariants, we get the algebra 
\[R_{(2,1)} = (\CP_{(2,1)})_{(1, 0)} \cong \C[u_1, u_2, v_1] / \left\langle u_1 + v_1, u_2+u_1v_1, u_2v_1-s \right\rangle \cong \C[u_1]/\left\langle u_1^3 \right\rangle, \]
%\[ \begin{array}{rcl}R_{(2,1)}\cong (\CP_{(2,1)})_{(1, 0)} & \cong &\C[u_1^{(1)}, u_1^{(2)}, u_2^{(1)}] \Big/ \left\langle u_1^{(1)}+ u_2^{(1)}, u_1^{(2)}+u_1^{(1)}u_2^{(1)}, u_1^{(2)}u_2^{(1)} \right\rangle\\ & \cong & \C[u_1^{(1)}]\Big/\left\langle \left(u_1^{(1)}\right)^3 \right\rangle, \end{array}\]
where to simplify notation, we have written $\{u_1, u_2, v_1\}$ for $\{u_1^{(1)}, u_1^{(2)}, u_2^{(1)}\}$. This graded algebra is isomorphic to 
the cohomology of the corresponding flag variety $F_{2,1}\cong \P^2$. For a general deformation, we obtain 
%\[\begin{array}{rcl} (\CP_{(2,1)})_{(1, s)} & \cong & \C[u_1^{(1)}, u_1^{(2)}, u_2^{(1)}] \Big/ \left\langle u_1^{(1)}+ u_2^{(1)}, u_1^{(2)}+u_1^{(1)}u_2^{(1)}, u_1^{(2)}u_2^{(1)}-s \right\rangle \\ & \cong &\C[u_1^{(1)}]\Big/\left\langle \left(u_1^{(1)}\right)^3+s \right\rangle, \end{array}\]
\[(\CP_{(2,1)})_{(1, s)} \cong \C[u_1, u_2, v_1] / \left\langle u_1 + v_1, u_2+u_1v_1, u_2v_1-s \right\rangle \cong \C[u_1]/\left\langle u_1^3+s \right\rangle, \]
which is indeed isomorphic to the (small) {\em quantum cohomology} $QH^*(\P^2)$ of $\P^2$, with quantum parameter $s$. 

%It can be checked that an analogous argument works for all $\alpha=(n-1,1)\vDash n$, with $F_\alpha\cong \P^{n-1}$: the one-dimensional deformation family of the algebra $P_\alpha$ considered above is, in this case, isomorphic to the quantum cohomology $qH^*(\P^{n-1})$. 

The bigraded dimension of the central fibre is 
\begin{equation} \label{eq:P2bigr}\dim_{t,q}(\CP_{(2,1)})_{(0, 0)} = \dim_{t,q} P_{(2,1)} = 1 + tq + tq^2, \end{equation}
which specialises at $t=1$ to the Poincar\'e polynomial
\[
\dim_{q}(\CP_{(2,1)})_{(q, 0)} = \dim_{q} R_{(2,1)} =  1 + q + q^2 =  \dim_{q}H^*(\P^2, \C). 
\]
The question is then whether there is a natural cohomological interpretation starting from $\P^2$ attached to the bigraded expression~\eqref{eq:P2bigr}.
\end{example}

\subsection{Garsia--Stanton--style bases in projective and partial coinvariant algebras}\label{sec:bases} 
%This last section is mostly a review, as the key result is already available in the literature in~\cite[Sect.~5.1]{ORS}. For $\alpha=(1^n)$, in other words for the case of permutations, equivalent results were obtained in~\cite[Sect.~3-4]{BO}. 

Fix $\alpha=(\alpha_1, \ldots, \alpha_k)\Vdash n$. Let ${\mathcal L}_\alpha$ denote the set
of increasing lattice paths from $(0,\ldots, 0)$ to $(\alpha_1,\ldots, \alpha_k)$ in $\N^k$, in other words the set of functions
$\ell\colon \{0,\ldots, n\}\to\N^k$ with the following properties: 
\begin{enumerate}
\item $\ell(0)=(0,\ldots, 0)$; 
\item $\ell(n)=(\alpha_1,\ldots, \alpha_k)$; 
\item for every $0\leq i<n$, we have $\ell(i+1)=\ell(i)+{\bf e}_j$ for some $1\leq j\leq k$, where 
$\{{\bf e}_1, \ldots, {\bf e}_k\}\subset\N^k$ are the standard unit vectors. 
\end{enumerate}

\begin{lemma} \label{lemma_paths} There is a natural bijection 
\[\delta_\alpha\colon W_\alpha\to{\mathcal L}_\alpha\] between the set of words $W_\alpha$ in the multiset $M_{\alpha}$ and the set of increasing lattice paths ${\mathcal L}_\alpha$. 
\end{lemma}
\begin{proof} A word $w\in W_\alpha$ in the multiset $M_{\alpha}=\{1^{\alpha_1}, \ldots, k^{\alpha_k}\}$ is an ordered list of numbers between $1$ and $k$, each number $j$ appearing $\alpha_j$ times. Such a word
is precisely an instruction how to build an increasing lattice path $\ell\in{\mathcal L}_\alpha$: in each step in (3), choose the $j$ that is the next number in the list. 
\end{proof}

Using this correspondence, we can define the set of descent positions ${\rm Des}(\ell)$ of a lattice path $\ell\in{\mathcal L}_\alpha$ to be the set of lattice points $\ell(i)\in \N^k$ for those $0\leq i <n$ that are descent positions of the corresponding word $w\in W_{\alpha}$. In other words, it is the set of lattice points $p\in \N^k$ on our path such that the step to this lattice point in the path has a higher value of $j$ than the step after it. 

Next, note that a lattice point $p=(i_1, \ldots, i_k)\in \N^k$ with $0\leq i_j\leq \alpha_j$ for all $j$, in other
words a lattice point that can be visited by one of our lattice paths, can be identified with an element
$I(p)=(i_1,\ldots, i_k)\in {\mathcal P}(M_\alpha)$, a subset of the multiset $M_\alpha$. Hence any such lattice
point $p$ can be mapped to a unique monomial $y_{I_p}$ in the coordinate ring $T_\alpha$ of a general Segre embedding, 
compare~\ref{sec:Talpha}. 

Given a word $w\in W_\alpha$ in the multiset $M_\alpha$,
%a lattice path $\ell \in {\mathcal L}_\alpha$
we can finally define an element
\[
y_w = \prod_{p\in {\rm Des}(\delta(w))} y_{I_p}\in T_\alpha.
\]
Looking at the family $\CP_\alpha$ discussed in Theorem~\ref{thm:Palpha_def}, we can consider the corresponding element
\[\mathbf{a}_w=[y_w]_{0,0}\in  (\CP_\alpha)_{(0,0)} \cong P_\alpha\] in the central fibre, the Artinian truncation $P_\alpha$ of $T_\alpha$, homogeneous of bidegree $(\des(w), \maj(w))$. We can also consider the deformed element
\[\mathbf{b}_w= [y_w]_{1,0}\in  (\CP_\alpha)_{(1,0)}\cong R_\alpha\]
in the partial coinvariant algebra, homogeneous of degree $\maj(w)$.

\begin{theorem}\!\!\! \cite[Thm.~5.2]{ORS} \ \label{thm_bases} 
The set of elements
\[\{ \mathbf{a}_w: w \in W_\alpha\}\subset P_\alpha
\]
is a bigraded basis of the finite-dimensional bigraded algebra $P_\alpha$, witnessing the Hilbert series identity~\eqref{eq:Patqdim}. 
\end{theorem}
The case $\alpha=(1^n)$ of this result, corresponding to the projective coinvariant algebra itself, was proved independently in \cite[Thm.~4.1]{BO}.

\begin{corollary}
The set of elements
\[\{ \mathbf{b}_w : w \in W_\alpha\}\subset R_\alpha
\]
is a graded basis of the finite-dimensional graded algebra $R_\alpha$, witnessing the Hilbert series identity~\eqref{eq:dimqRalpha}. 
\end{corollary}

\begin{proof} This follows from Theorem~\ref{thm_bases} together
with Theorem~\ref{thm:Palpha_def}, 
as being a basis in a flat family of finite-dimensional algebras is an open property. 
\end{proof}

\begin{example} Consider the case $\alpha=(1^n)$. In this case, elements of the basis in Theorem~\ref{thm_bases} are simply indexed by permutations, identified under Lemma~\ref{lemma_paths} with the set of increasing lattice paths in the unit hypercube. 
Checking through the definitions, the basis element $\mathbf{a}_\sigma\in P_n$ is the image under the quotient map $T_n\to P_n$ of the monomial
\[\prod_{i\in\Des(\sigma)} x_{\{\sigma(1),\ldots, \sigma(i)\}}\in T_n,  
\]
of bidegree $(\des(\sigma), \maj(\sigma))$.
%where $q_n\colon T_n\to P_n$ is the quotient map defining the projective coinvariant algebra.  
In the deformation $R_n$ of $P_n$, the corresponding element $\mathbf{b}_\sigma\in R_n$ is the image under the quotient map $\C[u_1,\ldots, u_n]\to R_n$ of the monomial
\[
\prod_{i\in\Des(\sigma)} \prod_{j=1}^i u_{\sigma(i)}\in \C[u_1,\ldots, u_n],
\]
of degree $\maj(\sigma)$. These elements form the standard Garcia--Stanton descent basis~\cite{GS} of $R_n$. 
\end{example}

For arbitrary $\alpha\Vdash n$, the Garcia--Stanton-style basis $\{\mathbf{b}_w : w \in W_\alpha\}$ of the partial coinvariant algebra~$R_\alpha$, indexed by words in a multiset, or equivalently, by Lemma~\ref{lemma_paths}, by increasing lattice paths, may be new. Starting from a lattice path, the basis elements can be expressed easily in the generating set of Proposition~\ref{prop:cohom}. Instead of writing long expressions, we restrict to a final example. 

\begin{example} Take $\alpha=(2,2)\vDash 4$, continuing Example~\ref{ex22}. Once again replacing $\{u_1^{(1)}, u_1^{(2)}, u_2^{(1)}, u_2^{(2)}\}$ with the more readable $\{u_1, u_2, v_1, v_2\}$, we have
\[\begin{array}{rcl} R_{(2,2)} & \cong & \C[u_1, u_2, v_1,v_2] / \left\langle u_1 + v_1, u_2+u_1v_1+v_2, u_2v_1+u_1v_2, u_2v_2 \right\rangle \\
& \cong & \C[u_1, u_2]/\left\langle u_1^3-2u_1u_2, u_1^2u_2-u_2^2\right\rangle,
\end{array}\]
a graded algebra with $\deg(u_i)=i$, $\deg(v_i)=i$ and $q$-dimension $1+q+2q^2 + q^3+q^4$,
isomorphic to the cohomology algebra $H^*({\rm Gr}(2,4), \C)$.
The set of words 
\[W_\alpha = \{1122, 1212, 1221, 2112, 2121, 2211\}\]
gives rise to the graded $\C$-basis
\[\{1, u_1v_1, u_1v_2, v_1, u_1v_1v_2, v_2\}\subset R_{(2,2)};\]
for each lattice path, the rule is to multiply monomials at the descent positions, as in Figure~1. In terms of the second presentation, this becomes
\[ \{1, -u_1^2, u_1u_2, -u_1,-u_2^2, u_1^2-u_2\}\subset R_{(2,2)}.\]
Even in this simple case, and disregarding signs, this basis is different both from the Schubert basis $\{1,u_1, u_1^2-u_2, u_2, u_1u_2, u_2^2\}$ and the standard monomial basis $\{1,u_1, u_1^2, u_2, u_1u_2, u_2^2\}$ of the algebra $R_{(2,2)}\cong H^*({\rm Gr}(2,4), \C)$.

%\[\begin{array}{rcl} R_{(2,2)} & \cong & \C[u_1, u_2, v_1,v_2] / \left\langle u_1 + v_1, u_2+u_1v_1+v_2, u_2v_1+u_1v_2, u_2v_2 \right\rangle \\
%& \cong & \C[v_1, v_2]/\left\langle v_1^3-2v_1v_2, v_1^2v_2-v_2^2\right\rangle,
%\end{array}\]
%\[ \{1, -v_1^2, -v_1v_2, v_1, v_2^2, v_2\}\subset R_{(2,2)}.\]

\begin{figure}[h]
\begin{center}
    \begin{tikzpicture}[scale=1.5, font=\footnotesize, fill=black!20]
	\draw (0,0) -- (2,0);
	\draw (0,1) -- (2,1);
	\draw (0,2) -- (2,2);
    \draw (0,0) -- (0,2);
	\draw (1,0) -- (1,2);
	\draw (2,0) -- (2,2);
    \draw[line width=2.5pt] (0,0) -- (0,1);
    \draw[line width=2.5pt] (0,1) -- (1,1);
    \draw[line width=2.5pt] (1,1) -- (1,2);
    \draw[line width=2.5pt] (1,2) -- (2,2);
    \draw[dotted, line width=2.5pt] (0,0) -- (1,0);
    \draw[dotted, line width=2.5pt] (1,0) -- (1,1);
    \draw[dotted, line width=2.5pt] (1,1) -- (2,1);
    \draw[dotted, line width=2.5pt] (2,1) -- (2,2);    
    \draw (-0.12,-0.12) node {$1$};
    \draw (0.85,-0.15) node {$u_1$};
    \draw (1.85, -0.15) node {$u_2$};
    \draw (-0.15,0.85) node {$v_1$};
    \draw (0.75,0.85) node {$u_1v_1$};
    \draw (1.75,0.85) node {$u_2v_1$};
    \draw (-0.15,1.85) node {$v_2$};
    \draw (0.75,1.85) node {$u_1v_2$};
    \draw (1.75,1.85) node {$u_2v_2$};
    \fill[black] (0,1) circle [radius=.1cm];
   \fill[black] (1,2) circle [radius=.1cm];
   \draw[dotted, line width=2.5pt] (1,1) circle [radius=.1cm];
    \end{tikzpicture}
     \caption{The lattice paths corresponding to the words $w=1212$ (dotted line) and $w=2121$ (solid line), with their descent positions marked, give rise to the basis elements $u_1v_1$ and $v_1\cdot u_1v_2=u_1v_1v_2$}
\end{center}\end{figure}
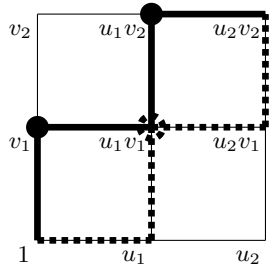
\end{example}

%More generally, for  $\alpha=(n-1,1)\vDash n$, we have $F_\alpha\cong \P^{n-1}$, while 
%\[\dim_{t,q} P_\alpha = 1 + t(q + q^2 + \ldots + q^{n-1}) \]
%specialises to 
%\[\dim_{q} R_\alpha = \dim_{q}H^*(\P^{n-1}, \C) = 1 + q + \ldots + q^{n-1}. \]

\end{document}